\documentclass[12pt]{article}
\usepackage{amssymb,amsbsy,bm,times,t1enc,url}

\ifx\pdfoutput\undefined 
\usepackage[dvips]{graphicx}
\RequirePackage[dvips]{color}
\else
\usepackage[pdftex]{graphicx}
\RequirePackage[pdftex]{color}
\fi

\usepackage{amsfonts}
\usepackage{amssymb}
\usepackage{bm}

\definecolor{darkgreen}{rgb}{0,0.3,0}

\newcommand{\pr}[1]{\mathbb{P}\hspace{-0.1em}\left(#1\right)}
\newcommand{\cPr}[2]{\mathbb{P}\hspace{-0.6mm}\left[\left.#1\,\right|#2\right]}
\newcommand{\E}[1]{\mathbb{E}\set{#1}}
\newcommand{\cE}[2]{\mathbb{E}\left[\left.#1\,\right|#2\right]}

\newcommand{\set}[1]{\left\{{#1}\right\}}
\newcommand{\eset}[2]{\left\{{#1} : \: {#2}\right\}}
\newcommand{\indic}[1]{1_{\{#1\}}}
\newcommand{\Nat}{\mathbb{N}}
\renewcommand{\Re}{\mathbb{R}}
\newcommand{\D}{{\,\mathrm{d}}}
\newenvironment{equation*}{\[}{\]}

\newtheorem{proposition}{Proposition}[section]

\newtheorem{theorem}[proposition]{Theorem}
\newtheorem{lemma}[proposition]{Lemma}

\newenvironment{proof}{{\it
Proof}\quad}{\nopagebreak\hspace*{\fill}$\square$\par
\bigbreak}

\begin{document}

\title{On the robustness of power-law random graphs in the finite
mean, infinite variance region}
\author{{Ilkka Norros and Hannu Reittu}\\
VTT Technical Research Centre of Finland}
\date{}
\maketitle

\begin{abstract}
We consider a conditionally Poissonian random graph model where the
mean degrees, `capacities', follow a power-tailed distribution with
finite mean and infinite variance. Such a graph of size $N$ has a
giant component which is super-small in the sense that the typical
distance between vertices is of the order of $\log\log N$. The
shortest paths travel through a core consisting of nodes with high
mean degrees. In this paper we derive upper bounds of the distance
between two random vertices when an upper part of the core is removed,
including the case that the whole core is removed.
\end{abstract}

\section{Introduction}\label{intro} 

Power-law random graphs (or `scale-free random graphs') have become
popular objects of both applied and theoretical interest, because they
are simple to define and generate and yet share some important
characteristics with many large and complex real-world networks. The
mathematical models help to understand how and on what conditions those
characteristics emerge, and this understanding can be valuable also in
the design of new, artificial networks.

The general characteristic of power-law random graphs is that their
degree distribution possesses a regularly varying tail, the most
interesting region of tail exponents being that of finite mean and
infinite variance. Several models with these features have been
studied. In this paper, we work with the conditionally Poissonian
random graph \cite{norrosreittu06}, a modification of the expected
degree sequence model of Chung and Lu \cite{chunglu03}. We draw an
i.i.d.\ sequence of mean degrees, `capacities', that follow a
Pareto distribution with finite mean and infinite variance.

Our graph of size $N$ has a giant component which is ultra-small in
the sense that the typical distance between vertices is proportional
to $\log\log N$ (more precisely, this is an upper bound, like most of
our distance results), whereas a similar random graph whose degrees
have the same mean but finite variance cannot offer better than $\log
N$ scalability of distances. Remarkably, the speciality of the
infinite variance case is the emergence of a `core network' consisting
of nodes with high degrees, through which the short paths typically
travel. It is now natural to ask, what happens to these distances, if
vertices with highest degrees are deleted. Our main findings can be
summarized as follows: 
\begin{enumerate}
\item Any deletion of core vertices leaves the asymptotical size of
the largest, giant component intact. 
\item If vertices with capacity higher than $N^\gamma$ are deleted,
$\gamma$ not depending on $N$, then paths that would have gone through
the deleted vertices can be repaired using `back-up paths'
that increase the distances only by a constant depending on $\gamma$
but not on $N$. This increase is negligible on the $\log\log N$
scaling of distances.
\item If the whole core is removed (for the exact meaning of this, see
Section \ref{modelsec}), the distances still scale slightly better
than $\log N$.
\end{enumerate}

Because of structural similarity (compare \cite{reittunorrosperfev04}
with \cite{norrosreittu06}), we conjecture that our results are
transferable to the configuration model of Newman {\em et al.}\
\cite{newman2001,reittunorrosperfev04}. Perhaps they can be extended
even to preferential attachment models in the finite mean and infinite
variance region, since van der Hofstad and Hooghiemstra showed
recently that these models share with the previously mentioned models
the loglog-scalability of distances \cite{hofstadhooghiemstra07}.

The next section specifies the graph model and reviews the relevant
results of \cite{norrosreittu06}. Some of them are slightly sharpened in
Section \ref{auxsec}. Section \ref{backupsec} applies results on the
diameter of classical random graphs to measure the core network in the
horizontal direction. The main result on robustness against core
losses is stated and proven in Section \ref{fullpathsec}.

\section{Model and earlier results}\label{modelsec}

Throughout the paper, we work with the following conditionally
Poissonian power-law random graph model \cite{norrosreittu06}. There
are $N$ vertices that possess, respectively, i.i.d.\ `capacities'
$\Lambda_1,\ldots,\Lambda_N$ with distribution
$$
\pr{\Lambda>x}=x^{-\tau+1},\quad x\in[1,\infty);\quad\tau\in(2,3).
$$
Given $\mathbf{\Lambda}=(\Lambda_1,\ldots,\Lambda_N)$, each pair
of vertices $\set{i,j}$, is connected with $E_{ij}$ edges, where
$$
E_{ij}\sim\mbox{Poisson}\left(\frac{\Lambda_i\Lambda_j}{L_N}\right),
\quad L_N=\sum_{j=1}^N\Lambda_j,
$$
and the $E_{ij}$'s are independent. Loops, i.e.\ the case $i=j$, are
included for principle, although they have no significance in the
present study. Given $\mathbf{\Lambda}$, the degree of vertex $i$ then
has the distribution $\mbox{Poisson}(\Lambda_i)$. The resulting random
graph is denoted as $G_N$.

The following coupling between neighborhood shells and branching
processes will play an important role, as it did in
\cite{norrosreittu06}. Fix $N$, and assume that the sequence
$\mathbf{\Lambda}$ has been generated and probabilities are now
conditioned on it. We shall consider a {\em marked branching process},
i.e., a branching process where each individual is associated with
some element of the mark space $\set{1,\ldots,N}$. More specifically,
let us define a process $(Z,\mathbf{J})=(Z_n,(J_{n,i}))$, where $Z_n$
is the size of generation $n$, and $J_{n,i}\in\set{1,\ldots,N}$ is the
mark of member $i$ of generation $n$. We set $Z_0\equiv1$ and take
$J_{0,1}$ from the uniform distribution $U\{1,\ldots,N\}$. The process
then proceeds so that an invidual bearing mark $i$ gives for each
$j=1,\ldots,N$ birth independently to a
Poisson($\Lambda_i\Lambda_j/L_N$) distributed number of $j$-marked
members of the next generation.

On the other hand, we consider the neighborhood sequence around a
random vertex of $G_N$. Take a vertex $i_0$ from uniform distribution,
and define recursively
\begin{eqnarray*}
\mathcal{N}_0(i_0)&=&\set{i_0}\\
\mathcal{N}_{k+1}(i_0)&=&
  \eset{j\in\left(\bigcup_{l=0}^k\mathcal{N}_l(i_0)\right)^c}{
  j\leftrightarrow\mathcal{N}_k(i_0)}.
\end{eqnarray*}
The following coupling was proven in \cite{norrosreittu06}.

\begin{proposition}\label{couplingprop1}
Let $(Z,\mathbf{J})$ be the marked branching process defined
above. Define a reduced process by proceeding generation by generation,
i.e., in the order
$$
J_{0,1};J_{1,1},J_{1,2},\ldots,J_{1,Z_1};J_{2,1},\ldots,
$$ and pruning (that is, deleting) from $(Z,\mathbf{J})$ each
individual, whose mark has already appeared before, together with all
its descendants (these are considered as not seen in the
procedure). Denote the resulting finite process by
$(\hat{Z},\hat{\mathbf{J}})$, and let $\hat{\mathcal{J}}_k$ be the set
of the marks in generation $k$ of the reduced process. Then, the
sequence of the sets $\hat{\mathcal{J}}_k$ has the same distribution
as the sequence $\mathcal{N}_k$.
\end{proposition}

Moreover, we observed that the size of each generation and its marks
can be generated independently. Denote by $q_N$ the distribution
\begin{equation}
\label{qdistdef}
q_N(j)\doteq\frac{\Lambda_j}{L_N},\quad j=1,\ldots,N,
\end{equation}
and by $\pi^*$ the mixed Poisson distribution Poisson($\Gamma$),
where $\Gamma$ is a random variable with distribution
$$
\pr{\Gamma\in\D x}=x\pr{\Lambda\in\D x}/\E{\Lambda}.
$$ 
Obviously, the distribution of Poisson($\Lambda_{J^{(N)}}$),
$J^{(N)}\sim q_N$, converges weakly to $\pi^*$ in probability.

Let $J_1^{(N)},J_2^{(N)},\ldots$ be i.i.d.\ random variables with
distribution $q_N$, and let $J_0^{(N)}$ be a random variable with
uniform distribution on $\set{1,\ldots,N}$, independent of the
previous ones; we often suppress the superscript $\cdot^{(N)}$. Start now
by $(\tilde{Z}_0=1,J_0)$, and proceed generationwise by drawing the
size of the $n+1$'th generation from the distribution
$$
\tilde{Z}_{n+1}\sim\mathrm{Poisson}
\left(\sum_{i=M_{n-1}+1}^{M_n}\Lambda_{J_i}\right),
$$
where $M_{-1}=0$, $M_n=\sum_0^n\tilde{Z}_k$, and giving then each
individual a fresh mark from the sequence
$J_{M_n+1},J_{M_n+2},\ldots$.

\begin{proposition}\label{couplingprop2}
The marked generation sequence $(Z_n,(J_{n,i}))$ is stochastically
equivalent to 
$$
(\tilde{Z}_n,(J_{M_{n-1}+1},\ldots,J_{M_n})).
$$
\end{proposition}

We next turn to the definition of the core network mentioned in the introduction.
Fix an increasing function $\ell:\Nat\to\Re$ such that
\begin{equation}
\label{elldef}
\ell(1)=1,\quad\frac{\ell(N)}{\log\log\log N}\to0,\quad
\frac{\ell(N)}{\log\log\log\log N}\to\infty
\quad\mathrm{as\ }N\to\infty,
\end{equation}
and define
$$
\epsilon(N):=\frac{\ell(N)}{\log N}.
$$
Note that
\begin{equation}
\label{Nepstoinfty}
\epsilon(N)\to0,\quad
N^{\epsilon(N)}=e^{\ell(N)}\to\infty\quad\mathrm{as\ }N\to\infty.
\end{equation}

Define recursively the functions
\begin{eqnarray}
\nonumber
\beta_0(N)&=&\frac1{\tau-1}+\frac{\epsilon(N)}{\tau-2},\\
\label{betakdef}
\beta_j(N)&=&(\tau-2)\beta_{j-1}(N)+\epsilon(N),\quad j=1,2,\ldots,
\end{eqnarray}
and denote
$$
k^*\doteq \left\lceil\frac{\log\log N}{-\log(\tau-2)}\right\rceil.
$$
For $N$ sufficiently large, we have $\beta_0(N)>\epsilon(N)/(3-\tau)$,
and the sequence $(\beta_k(N))_{k=0,1,\ldots}$ decreases toward
the limit value $\epsilon(N)/(3-\tau)$. With $k=k^*$, we are already
in the $\epsilon(N)$ order of magnitude:
\begin{equation}
\label{beta_k*-ub}
\beta_{k^*(N)}(N)\le\epsilon(N)\left(1+\sum_{i=0}^\infty(\tau-2)^i\right)
=\frac{4-\tau}{3-\tau}\,\epsilon(N).  
\end{equation}

The key of our analysis of $G_N$ is the notion of its {\em core} $C$,
consisting of all vertices with capacity higher than
$N^{\beta_{k^*}}$. Note that the exact boundary of the core depends on
$\ell(N)$ and is thus somewhat arbitrary. 

By the definition of $\epsilon(N)$, we can choose a natural
number $\kappa=\kappa(N)$ such that
\begin{equation}
\label{kappa-defcond}
\frac{\kappa(N)}{N^{\theta\epsilon(N)}}\to\infty,\quad
\frac{\kappa(N)}{k^*(N)}\to0,\quad\mbox{as }N\to\infty,
\end{equation}
where $\theta=(\tau-2)(4-\tau)/(3-\tau)$. We can now collect the main
results of \cite{norrosreittu06} to the following theorem:

\begin{theorem}\label{loglogthm}
\begin{enumerate}
\item\label{giantclaim} The graph $G_N$ has a.a.s.\ (asymptotically
almost surely) a giant component,
whose relative size approaches the value
\begin{equation}
\label{giantsize2}
1-\sum_{j=1}^\infty\pr{D=j}\pr{Z^{(\pi^*)}_\infty=0}^j,
\end{equation}
where $D$ is distributed as the conditionally Poissonian variable
Poisson($\Lambda$), and $(Z^{(\pi^*)}_n)$ is a Galton-Watson branching
process with offspring distribution $\pi^*$.
\item\label{corefoundclaim} A random vertex of the giant component is
connected with the core $C$ in less than $\kappa(N)$ hops. The
distribution of the capacity of the first core vertex found in a
random-order breadth-first neighborhood search is asymptotically
identical to the distribution of $J^{(N)}_{n_C}$, where
$n_C=\inf\eset{n}{J^{(N)}_n\in C}$.
\item\label{coreheightclaim} A random vertex of the core is a.a.s.\
connected with the vertex $i^*$ with highest capacity in at most $k^*$
hops.
\item\label{typdistanceclaim} As a consequence of the previous items,
the distance between two randomly chosen vertices of the giant
component is at most $2k^*(N)(1+o(1))$ a.a.s..
\end{enumerate}
\end{theorem}

Denote $\mathcal{T}_{\gamma,\delta}
=\eset{i\in\set{1,\ldots,N}}{\Lambda_i\in(N^\gamma,N^\delta]}$. The
idea of the proof of claim \ref{coreheightclaim} of Theorem
\ref{loglogthm} is that if the core is divided into `tiers'
\begin{eqnarray*}
V_0&=&\set{i^*},\\
V_1&=&\mathcal{T}_{\beta_1,\infty}\setminus\set{i^*},\\
V_k&=&\mathcal{T}_{\beta_k,\beta_{k-1}},\quad k=2,\ldots,k^*,
\end{eqnarray*}
then, for any $k>0$, a random vertex of tier $k$ has a.a.s.\ an edge
to $V_{k-1}\cup\cdots\cup V_0$. 

\section{Auxiliary results}\label{auxsec}

In this section we sharpen the estimate used in \cite{norrosreittu06}
for the aggregated capacity of a set of form
$\eset{i}{\Lambda_i>N^\gamma}$. In particular, we obtain that the
aggregated capacity of a tier
$\mathcal{T}_{\gamma,\gamma+\epsilon(N)}$ is asymptotically
overwhelmingly larger than all strictly higher vertex capacities
together. This holds also when $\gamma$ depends on $N$. 

We use frequently the notation $x+I:=\eset{x+y}{y\in I}$, where
$x\in\Re$ and $I\subset\Re$ is an interval.

\begin{lemma}\label{uppersetvolumelemma}
Let $\alpha=\alpha(N)\in(0,\bar{\alpha})$, where 
$\bar{\alpha}<1/(\tau-1)$. Then, for $N$ sufficiently large,
$$
\pr{\frac{\sum_{i=1}^N\Lambda_i\indic{\Lambda_i>N^{\alpha}}}{
N^{1-(\tau-2)\alpha}\E{\Lambda}}\in(\frac14,4)}
\ge 1-\left(\frac{\E{\Lambda}}{2}\right)^{-\tau+1}
  N^{-(1-(\tau-1)\bar{\alpha})(\tau-2)}. 
$$
\end{lemma}
\begin{proof}
Denote
$$
p(n)=\pr{\frac{\frac1n\sum_{i=1}^n\Lambda_i}{\E{\Lambda}}
  \not\in(\frac12,2)},\quad
  N_\alpha=\sum_1^n\indic{\Lambda_i>N^\alpha}.
$$ 
Since
$\cPr{\Lambda>y}{\Lambda>x}=\pr{x\Lambda>y}$, we have the distribution
equality
\begin{equation}
\label{upsetlemeq1}
\sum_{i=1}^N\Lambda_i\indic{\Lambda_i>N^{\alpha}}
\buildrel\mathcal{D}\over=N^{\alpha}\sum_{k=1}^{N_\alpha}\tilde{\Lambda}_k,
\end{equation}
where the $\tilde{\Lambda}_k$'s are fresh independent copies of
$\Lambda$. Then
\begin{eqnarray}
\label{upsetlemeq2}
\nonumber
&&\pr{\frac{\frac{1}{N_\alpha}\sum_{k=1}^{N_\alpha}\tilde{\Lambda}_k}{\E{\Lambda}}
  \not\in[\frac12,2]}
=\E{p(N_{\alpha})}\\
&\le&\pr{N_\alpha<\frac12\E{N_\alpha}}+\E{p(N_\alpha)
  \indic{N_\alpha\ge\frac12\E{N_\alpha}}}.
\end{eqnarray}
The distribution of $N_\alpha$ is $Bin(N,N^{-(\tau-1)\alpha})$. The
well-known bounds for a binomial random variable $X$,
\begin{eqnarray*}
\pr{X<\E{X}-x}&\le&\exp(-\frac{x^2}{2\E{X}}),\\
\pr{X>\E{X}+x}&\le&\exp(-\frac{x^2}{2\E{X}}+\frac{x^3}{\E{X}^3}),\\
\end{eqnarray*}
yield
\begin{equation}
\label{upsetlemeq3}
\pr{\frac{N_\alpha}{\E{N_\alpha}}<\frac12}\le e^{-\E{N_\alpha}/8},\quad
\pr{\frac{N_\alpha}{\E{N_\alpha}}>2}\le e^{-\E{N_\alpha}/8+1}.
\end{equation}
By the well-known result on subexponential variables with finite mean
(see, e.g., \cite{adlerfeldmantaqqu98}), 
$$
\pr{\frac1n\sum_{i=1}^n\Lambda_i>2\E{\Lambda}}
\sim\pr{\max(\Lambda_1,\ldots,\Lambda_n)>n\E{\Lambda}}
\sim\E{\Lambda}^{-\tau+1}n^{-\tau+2}.
$$
For the large deviation to the opposite direction, the classical
Cram\'er theorem applies and the probability goes to zero at
exponential speed. Thus, for $N$ sufficiently large,
\begin{equation}
\label{upsetlemeq4}
\E{p(N_{\alpha})}\le e^{-\E{N_\alpha}/8}
+\frac32\E{\Lambda}^{-\tau+1}\left(\frac{\E{N_\alpha}}{2}\right)^{-\tau+2}.
\end{equation}
Combining (\ref{upsetlemeq1})-(\ref{upsetlemeq4}), replacing $\alpha$
by the worse case $\bar{\alpha}$ and upperbounding
the exponential terms by replacing the factor $3/2$ in
(\ref{upsetlemeq4}) by 2 we get the assertion.
\end{proof}

\begin{lemma}\label{tiervolumelemma}
Let $\alpha_0(N),\alpha_1(N)\in(0,\bar{\alpha})$, where
$\bar{\alpha}<1/(\tau-1)$, and
$\alpha_0(N)\le\alpha_1(N)-\epsilon(N)$. Then, for $N$
sufficiently large,
$$
\pr{\frac{\sum_{i=1}^N\Lambda_i\indic{\Lambda_i\in(N^{\alpha_0},N^{\alpha_1}]}}{
N^{1-(\tau-2)\alpha_0}\E{\Lambda}}\in(\frac15,5)}
\ge 1-2\left(\frac{\E{\Lambda}}{2}\right)^{-\tau+1}
  N^{-(1-(\tau-1)\bar{\alpha})(\tau-2)}.
$$
\end{lemma}
\begin{proof}
This follows by applying Lemma \ref{uppersetvolumelemma} to both
$\alpha_0$ and $\alpha_1$ and noting that $N^{-(\tau-2)(\alpha_1-\alpha_0)}\to0$.
\end{proof}

\begin{lemma}\label{lowtierbiglemma}
For any fixed $b>1$,
$$
\frac{\sum_{i=1}^N\Lambda_i\indic{\Lambda_i>N^{b\epsilon(N)}}}{
\sum_{i=1}^N\Lambda_i\indic{\Lambda_i>N^{\epsilon(N)}}}\to0
\quad\mbox{in probability.}
$$
\end{lemma}
\begin{proof}
By Lemma \ref{uppersetvolumelemma}, 
$$
\frac{\sum_{i=1}^N\Lambda_i\indic{\Lambda_i>N^{b\epsilon(N)}}}{
\sum_{i=1}^N\Lambda_i\indic{\Lambda_i>N^{\epsilon(N)}}}
\cdot N^{(b-1)(\tau-2)\epsilon(N)}
\in[\frac{1}{16},16]\quad\mbox{a.a.s.,}
$$
and the claim follows since $N^{\epsilon(N)}\to\infty$.
\end{proof}

Lemma \ref{lowtierbiglemma}, combined with our earlier results in
\cite{norrosreittu06}, entails the important observation that the
first contact to the core is a.a.s.\ close to its bottom:

\begin{proposition}\label{findcoreprop}
Fix $b>0$. Let $I_N$ be a random node of the giant component of
$G_N$. Asymptotically almost surely, $I_N$ is connected to
$\mathcal{T}_{\beta_{k^*},b\beta_{k^*}}$ with a path that avoids the
set $\eset{i}{\Lambda_i>N^{b\beta_{k^*}}}$, and whose length coincides
with the distance between $I_N$ and the core.
\end{proposition}
\begin{proof}
Since the proportional size of the core shrinks to zero, it suffices
to consider the case that $I_N$ is picked from outside the core. By
the results in \cite{norrosreittu06}, conditional that $I_N$ is
connected with the core, one of the minimal length paths to the core
has its core endpoint distributed as the first core element in an
i.i.d.\ sequence $J_n\in\set{1,\ldots,N}$ with common distribution
(\ref{qdistdef}). Now, $\beta_{k^*}$ is proportional to $\epsilon(N)$,
and Lemma \ref{lowtierbiglemma} can be applied with $\beta_{k^*}$ in
the role of $\epsilon(N)$. We conclude that the first core element in
the sequence $J_n$ belongs a.a.s.\ to the set
$\mathcal{T}_{\beta_{k^*},b\beta_{k^*}}$, and the statement of the
proposition follows.
\end{proof}

\section{Horizontal paths in the core}\label{backupsec}

In this section, we show that quite thin tiers of the form
$\mathcal{T}_{\gamma',\gamma''}$ are internally (as induced subgraphs)
almost connected in the sense that the proportional size of the
largest component approaches one and, moreover, the diameter of that
largest component is obtained with high accuracy from remarkable
results on classical random graphs. For deterministic $\gamma$, the
tiers $\mathcal{T}_{\gamma-\epsilon(N),\gamma}$ are a.a.s.\ connected,
and their diameters can be picked almost unequivocally from the
following classical theorem \cite{bollobas2001}.

\begin{theorem}
\label{bollobasthm}
Consider the Erd\"os-R\'enyi random graph $G_{n,p}$ and let $p=p(n)$ and
$d=d(n)>2$ satisfy
\begin{eqnarray*}
 \frac{\log n}{d}-3\log\log n&\to&\infty\\
 p^dn^{d-1}-2\log n&\to&\infty\\
 p^{d-1}n^{d-2}-2\log n&\to&-\infty
\end{eqnarray*}
Then $G_{n,p}$ has diameter $d$ a.a.s.
\end{theorem}

To include also lower parts of the core, we apply the following,
rather recent result by Chung and Lu \cite{chunglu01}.

\begin{theorem}
\label{chungluthm}
If $np\to\infty$ and $(\log n)/(\log np)\to\infty$, then almost surely
$$
\mathrm{diam}(G_{n,p})=(1+o(1))\frac{\log n}{\log np},
$$
where $\mathrm{diam}(G)$ denotes the diameter of the largest (giant) component
of $G$. 
\end{theorem}

Denote
\begin{equation}
\label{w(gamma)def}
w(\gamma):=\left\lceil\frac{1-(\tau-1)\gamma}{(3-\tau)\gamma}
\right\rceil.
\end{equation}

\begin{proposition}\label{thintierdiamprop}
Let $\gamma(N)$ be a non-increasing function such that
$\gamma(N)\in[\epsilon(N),\frac12)$ for all $N$. If $\lim\gamma(N)>0$,
assume also that $\gamma(\infty)$ does not correspond to a jump
of the ceiling function in (\ref{w(gamma)def}). Then, there exists another
function $\gamma_*(N)$ such that 
$$
\gamma_*(N)<\gamma(N),\quad
\gamma_*(N)/\gamma(N)\to1,\quad
|\mathcal{T}_{\gamma_*(N),\gamma(N)}|\to\infty,
$$
and
\begin{equation}
\label{tierdiam}
\mathrm{diam}(\mathcal{T}_{\gamma_*,\gamma})
=(1+o(1))w(\gamma(N))\quad\mbox{a.a.s.},
\end{equation}
where $\mathrm{diam}(\cdot)$ is generally interpreted as in Theorem
\ref{chungluthm}. When $\gamma(N)$ is bounded away from zero,
$\mathcal{T}_{\gamma_*,\gamma}$ is a.a.s.\ connected and its diameter is
a.a.s.\ equal to $w(\gamma(N))$. 
\end{proposition}

\begin{proof}
We shall suppress the $N$'s in $\gamma(N)$ etc.\ when it improves
the clarity of presentation. 

Considered as an induced subgraph, $\mathcal{T}_{\gamma_*,\gamma}$ is
denser (resp.\ sparser) than the classical random graph $G_{n,p_*}$
(resp.\ $G_{n,p^*}$) with
$$
n=n(N)=|\mathcal{T}_{\gamma_*,\gamma}|,\quad
p_*=p_*(N)=\frac{N^{2\gamma_*}}{L_N},\quad
p^*=p^*(N)=\frac{N^{2\gamma}}{L_N}.
$$
Denote $\gamma_\delta(N)=\gamma-\delta\epsilon(N)$, where
$\delta\in(0,1)$. We fix $\delta$ for a while and choose
$\gamma_*=\gamma_\delta$. By Lemma \ref{tiervolumelemma},
$|\mathcal{T}_{\gamma_*,\gamma}|\to\infty$, and
$$
\frac{|\mathcal{T}_{\gamma_*,\gamma}|}{N^{1-(\tau-1)\gamma_*}}
\in[\frac15,5],\quad\mbox{a.a.s.}
$$
On the other hand, 
$$
\frac{L_N/N}{\E{\Lambda}}\in[\frac12,2]\quad\mbox{a.a.s.}
$$
Thus, a.a.s.,
\begin{eqnarray*}
\log n&\in&(1-(\tau-1)\gamma_*)\log N +[-\log 4,\log 4],\\
\log(np_*)&\in&(3-\tau)\gamma_*\log N-\log\E{\Lambda}+
[-\log 8,\log 8],\\
\log(np^*)&\in&(2\gamma-(\tau-1)\gamma_*)\log N-\log\E{\Lambda}+
[-\log 8,\log 8].
\end{eqnarray*}
Note that $p^*(N)\to0$ and $n(N)p_*(N)\to\infty$, so Theorem
\ref{chungluthm} applies to $G_{n,p_*}$ and $G_{n,p^*}$. We
distinguish between two cases. 

First, if $\gamma(N)$ is bounded away from zero, we are in fact in the
regime of Theorem \ref{bollobasthm}, and a computation of the diameter
$d$ leads to the expression of $w(\gamma)$. Thus, Theorem
\ref{bollobasthm} yields the inequalities
\begin{eqnarray*}
\mathrm{diam}(\mathcal{T}_{\gamma_*,\gamma})
&\le&\left\lceil\frac{1-(\tau-1)\gamma_*}{(3-\tau)\gamma_*}
  \right\rceil\quad\mbox{a.a.s.},\\
\mathrm{diam}(\mathcal{T}_{\gamma_*,\gamma})
&\ge&\left\lceil\frac{1-(\tau-1)\gamma_*}{2\gamma+(1-\tau)\gamma_*}
  \right\rceil\quad\mbox{a.a.s.}
\end{eqnarray*}
Now, the right hand sides are with sufficiently large $N$ both equal
to the right hand side of $w(\gamma)$, and the claim follows. The
value of $\delta$ did not matter.

The second case is that $\gamma(N)\searrow0$. Since
$n(N)p_*(N)\to\infty$, the subgraph still has a giant component whose
proportional size approaches one, although it is not necessarily
connected. Choose any sequence $\delta_k\searrow0$, and denote
$$
\eta_k=\max\left\{\frac{(3-\tau)\gamma}{(3-\tau)\gamma_{\delta_k}},
\frac{2\gamma+(\tau-1)\gamma_{\delta_k}}{(3-\tau)\gamma}\right\}.
$$
Note that $\eta_k\to1$. Now define a sequence $N_k$ as
\begin{eqnarray*}
N_k&=&\inf\bigg\{M:\ \pr{|\mathcal{T}_{\gamma_{\delta_k},\gamma}|
  \frac{N^{2\gamma_{\delta_k}}}{L_N}\ge k}\ge1-1/k\quad \forall N\ge M,\\
&&\pr{\mathrm{diam}(\mathcal{T}_{\gamma_*,\gamma})\in
  (\eta_k^{-2},\eta_k^2)\cdot\frac{1}{(3-\tau)\gamma}}\ge1-1/k\quad 
  \forall N\ge M\bigg\}.
\end{eqnarray*}
Thanks to Theorem \ref{chungluthm}, the $N_k$'s are finite numbers,
and obviously they grow unboundedly. Finally, denote 
$$
k(N):=\max\eset{k}{N_k\le N}
$$ 
and choose $\gamma_*(N)=\gamma_{\delta_{k(N)}}$. Then obviously
$$
\mathrm{diam}(\mathcal{T}_{\gamma_*,\gamma})=
(1+o(1))\frac{1}{(3-\tau)\gamma(N)}\quad\mbox{a.a.s.}
$$
\end{proof}

In a similar way, we can prove a diameter result for the tiers
$V_k=V_k^{(N)}$, defined at the end of Section \ref{modelsec}:

\begin{proposition}\label{coretierdiamprop}
The tiers $V_0,\ldots,V_{k^*}$, considered as induced subgraphs of
$G_N$, are almost connected in the sense that the relative sizes of their
largest components approach 1 as $N\to\infty$, and their diameters
satisfy, a.a.s.,
\begin{eqnarray*}
\mathrm{diam}(V_0)&=&0,\\
\mathrm{diam}(V_1)&=&2,\\
\mathrm{diam}(V_k)&\in&(1+o(1))[w(\beta_{k-1}),w(\beta_k)],\quad k=2,\ldots,k^*.
\end{eqnarray*}
\end{proposition}

By Proposition \ref{thintierdiamprop}, we can say that the width of
the core at height $N^\gamma$, measured by hop distance, is
$w(\gamma)$. Moreover, this holds also for varying $\gamma=\gamma(N)$
downto the bottom of the core. 

When $\gamma\in(0,\frac12)$, we have
$$
w((\tau-2)\gamma)>w(\gamma)+2.
$$
This has the following heuristic consequence. For connecting two
vertices with capacities $\approx N^{\gamma}$, a `horizontal' path
staying at about the same height in the core is longer than a path
that jumps at both ends with one hop to the height
$N^{(\tau-2)\gamma}$ and finds the horizontal connection at that
level. On the other hand, this is how high in the core the
neighborhood of a vertex typically reaches. Thus, a horizontal move
should be made at as high level as possible. 

\section{The main result}\label{fullpathsec}

We have now essentially collected all elements needed to understand
the situation where a top part of the core is deleted. First, since
the aggregated capacity of any top part is concentrated at its bottom,
the remains of the core are found from outside almost as easily as the
original core was found. Second, if a lower part of the core is alive,
we can find `vertical' paths from the bottom of the core to a top tier
of its remaining part, as we did with the undamaged core in
\cite{norrosreittu06} (a complete proof of this step requires a
slight modification of the proof of Proposition 4.3 of
\cite{norrosreittu06} and is presented below in detail). Third,
Proposition \ref{thintierdiamprop} guarantees the existence of a
`horizontal' path between almost any pair of vertices within this top
tier, and provides also an explicit upper bound for its length. It
remains to couple these pieces together.

\begin{theorem}\label{robustnessthm}
Let $\gamma:=\gamma(N)\in(\epsilon(N),\frac12)$ be such that
$\gamma(N)$ is non-increasing and $\gamma(N)/\epsilon(N)$
non-decreasing w.r.t.\ $N$. Denote by $H_\gamma=H^{(N)}$ the graph
obtained from $G_N$ by deleting all vertices with capacity higher than
$N^\gamma$, together with the edges attached to them. The following
hold:
\begin{enumerate}
\item\label{giantstableclaim} The graph $H_\gamma$ has a.a.s.\ a giant
component whose relative size approaches the relative size of the
giant component of $G_N$.
\item\label{prolongationclaim}  
The distance of two randomly chosen vertices of the giant component
of $H_\gamma$ is a.a.s.\ less than 
\begin{equation}
\label{asymptdistance}
(1+o(1))\left(\frac{2}{-\log(\tau-2)}\left(
\log\log N-\log\frac{1}{\gamma}\right)
+w(\gamma)\right),
\end{equation}
where $w(\gamma)$ is given in (\ref{w(gamma)def}).
\end{enumerate}
\end{theorem} 

\begin{proof}
$1^{\mathrm{o}}$. Denote
\begin{eqnarray*}
\gamma_0&=&\gamma,\\
\gamma_1&=&\gamma_0-\epsilon(N),\\
\gamma_{k+1}&=&(\tau-2)\gamma_k+\epsilon(N),\quad k=1,2,\ldots,
\end{eqnarray*}
and let 
$$
\bar{k}=\min\eset{k\ge1}{\gamma_k\le\frac{4-\tau}{3-\tau}\epsilon(N)}.
$$
It is easy to check (see $4^{\mathrm{o}}$ below) that
$\bar{k}=O(\log\log N)$. 

Consider the tiers $\mathcal{T}_{\gamma_1,\gamma_0}$,
$\mathcal{T}_{\gamma_2,\gamma_1}$,\ldots,
$\mathcal{T}_{\gamma_{\bar{k}},\gamma_{\bar{k}-1}}$. The idea of the
proof is to apply the diameter result, Theorem \ref{chungluthm}, on
$\mathcal{T}_{\gamma_1,\gamma_0}$, and, if $\bar{k}\ge2$, to imitate
our proof of Theorem \ref{loglogthm} to show that vertices of tier
$\mathcal{T}_{\gamma_{\bar{k}},\gamma_{\bar{k}-1}}$ are with high
probability connected to $\mathcal{T}_{\gamma_1,\gamma_0}$ with a path
jumping from tier to tier.

$2^{\mathrm{o}}$. Assume that $\bar{k}\ge2$. Let $I_0=I_0(N)$ be a
stochastic (not necessarily uniformly random) vertex of
$\mathcal{T}_{\gamma_{\bar{k}},\gamma_{\bar{k}-1}}$ such that,
conditioned on $\mathbf{\Lambda}$, $I_0$ is independent of the edges
within
$\bigcup_1^{\bar{k}}\mathcal{T}_{\gamma_j,\gamma_{j-1}}$. Define the
sequence of vertices $I_1$, $I_2$,\ldots recursively as follows. If
$I_n\in\mathcal{T}_{\gamma_{\bar{k}-n},\gamma_{\bar{k}-n-1}}$, then
$I_{n+1}$ is a neighbor (say, with smallest index) of $I_n$ belonging
to $\mathcal{T}_{\gamma_{\bar{k}-n-1},\gamma_{\bar{k}-n-2}}$ if such a
neighbor exists. Otherwise $I_{n+1}=I_n$, and the rest of the sequence
repeats $I_n$ as well.  Denote
\begin{eqnarray*}
A_n&=&\set{I_n\in\mathcal{T}_{\gamma_{\bar{k}-n},\gamma_{\bar{k}-n-1}}},
\quad n=0,1,\ldots,\bar{k}-1,\\
B'&=&\set{\frac{L_N}N\le2\E{\Lambda}},\\
B_k&=&\set{L(\mathcal{T}_{\gamma_{k+1},\gamma_{k}})
  \in\E{\Lambda}N^{1-(\tau-2)\gamma_k}(\frac15,5)},
 \quad k=0,\ldots,\bar{k}-1,\\
&&\mbox{where }L(S):=\sum_{i\in S}\Lambda_i,\\
B&=&B'\cap B_0\cap\cdots\cap B_{\bar{k}-1}.
\end{eqnarray*}
We have $A_{\bar{k}-1}\subseteq A_{\bar{k}-2}\subseteq\cdots \subseteq
A_0$. Proposition \ref{tiervolumelemma} yields
$$
\pr{B_0\cap\cdots\cap B_{\bar{k}-1}}\ge
1-2\bar{k}\left(\frac{\E{\Lambda}}{2}\right)^{-\tau+1}
  N^{-(1-(\tau-1)\gamma)(\tau-2)}\to1
$$
as $N\to\infty$. On $B$ in turn, we have for each $n$
by Proposition \ref{tiervolumelemma} that
\begin{eqnarray*}
\cPr{A_{n+1}^c}{\mathbf{\Lambda},A_n}
&=&\cE{\exp\left(-\frac{\Lambda_{I_n}
  L(\mathcal{T}_{\gamma_{\bar{k}-n-1},\gamma_{\bar{k}-n-2}})}{L_N}
  \right)}{\mathbf{\Lambda},A_n}\\
&\le&\exp\left(-\frac1{2\E{\Lambda}N}N^{\gamma_{\bar{k}-n}}\cdot\frac1{5}
    N^{1-(\tau-2)\gamma_{\bar{k}-n-1}}\right)\\
&\le&\exp\left(-\frac{e^{\ell(N)}}{10\E{\Lambda}}\right).
\end{eqnarray*}
It follows that on $B$
\begin{eqnarray*}
\cPr{A_{\bar{k}-1}}{\mathbf{\Lambda},A_0}&=&
  \prod_{n=0}^{\bar{k}-1}\cPr{A_{n+1}^c}{\mathbf{\Lambda},A_n}\\
&\ge&\left(1-\exp\left(-\frac{e^{\ell(N)}}{10\E{\Lambda}}\right)\right)^{\bar{k}}\\
&\ge&1-\bar{k}\exp\left(-c e^{\ell(N)}\right)\\
&=&1-\frac{\bar{k}}{\log^2 N}\exp\left(\log^3N-c(e^{\log^4N})^{\ell(N)/\log^4N}
  \right)\\
&=&1-\frac{\bar{k}}{\log^2 N}
  \exp\left(-\log^3N\left(c(\log^3N)^{\ell(N)/\log^4N-1}-1\right)\right)\\
&\to&1\quad\mbox{as }N\to\infty
\end{eqnarray*}
by the assumption that $\ell(N)/\log^4N\to\infty$.

$3^{\mathrm{o}}$. Now let $I$ and $I'$ be two independent uniformly
randomly chosen vertices of $G_N$. By Proposition \ref{findcoreprop},
$I$ is a.a.s.\ either outside the largest component of $G_N$, or it is
connected to $\mathcal{T}_{\gamma_{\bar{k}},\gamma_{\bar{k}-1}}$ with
a path that avoids the set $\eset{i\le
N}{\Lambda_i>N^{\gamma_{\bar{k}-1}}}$ and whose length is at most
$\kappa(N)$. Denote the endpoint of that path by $I_0$. In the case
that $\bar{k}\ge2$, step $2^{\mathrm{o}}$ above showed that $I_0$ is
a.a.s.\ connected to a vertex
$I_{\bar{k}-1}\in\mathcal{T}_{\gamma_1,\gamma_0}$ over a path of
length $\bar{k}$. The same claims holds for $I'$, with vertices $I_0'$
and $I_{\bar{k}-1}'$ respectively. Now, $I_{\bar{k}-1}$ and
$I_{\bar{k}-1}'$ belong a.a.s.\ to the largest component of the
induced subgraph $\mathcal{T}_{\gamma_1,\gamma_0}$, and Theorem
\ref{chungluthm} yields that their distance is at most
$(1+o(1))w(\gamma)$.

$4^{\mathrm{o}}$. We have now shown that if $I$ and $I'$ both belong
to the largest component of $G_N$, their distance is a.a.s.\ at most 
$$
(1+o(1))(2\bar{k}+w(\gamma)).
$$
It remains to show that $\bar{k}$ can in this expression be replaced
by
\begin{equation}
\label{barksubstitute}
\frac{\log\log N-\log(1/\gamma)}{-\log(\tau-2)}.
\end{equation}
First, if $\gamma(N)\le C\epsilon(N)$ for all $N$ and some finite
$C>1$, then
$$
\log\frac{1}{\gamma}\in (\log\log N-\log\ell(N))+[-\log C,0],
$$ 
and (\ref{barksubstitute}) is in this region negligible compared with
$w(\gamma)$. On the other hand, $\bar{k}$ is in this region limited by
a constant, whereas $w(\gamma(N))\to\infty$. Thus, the assertion of
the theorem holds.

Assume then that $\gamma(N)/\epsilon(N)$ is unbounded. Then
$\gamma_1>((4-\tau)/(3-\tau))\epsilon(N)$ for large $N$, so that
$\bar{k}\ge2$. Since
$$
\gamma_{k+1}=(\tau-2)^k\gamma_1+\epsilon(N)\sum_{i=0}^{k-1}(\tau-2)^i,
$$
a short computation (using the geometric series formula) yields that
the condition $\gamma_k\le((4-\tau)/(3-\tau))\epsilon(N)$ is
equivalent to the condition
$$
(\tau-2)^k\left(\gamma_1-\frac{\epsilon(N)}{3-\tau}\right)\le\epsilon(N),
$$
and further to
\begin{eqnarray*}
k&\ge&\frac{1}{-\log(\tau-2)}\left(\log\log N
  +\log\left(\gamma(N)-\frac{4-\tau}{3-\tau}\epsilon(N)\right)
  -\log\ell(N)\right)\\
&=&(1+o(1))\frac{\log\log N-\log(1/\gamma(N))}{-\log(\tau-2)}.
\end{eqnarray*}
Thus, $\bar{k}$ equals the last, asymptotic expression. This finishes the proof.
\end{proof}

Theorem \ref{robustnessthm} tells that removing a top part of the
core downto $N^\gamma$ with fixed $\gamma>0$ has no effect to the
asymptotic upper bound (\ref{asymptdistance}). On the other hand, with
$\gamma(N)\propto\epsilon(N)$ we obtain
$$
w(\gamma(N))\propto\frac{1}{(3-\tau)\epsilon(N)}=\frac{\log
N}{(3-\tau)\ell(N)}.
$$
Combining this with Proposition \ref{findcoreprop}, we find that if
the whole core is removed, the remaining graph still has a giant
component with same proportional size as the original and diameter
proportional to $\log(N)/\ell(N)$ --- slightly smaller than that of a
(supercritical) power-law graph with $\tau>3$. Asymptotics of
intermediate cases between these two extremes can be computed from
(\ref{findcoreprop}) as well.

\bibliographystyle{plain}
\bibliography{/home/norros/biblio/mat,/home/norros/biblio/atm}

\end{document}